\documentclass[reqno]{amsart}

\usepackage{amsmath,amsfonts,amssymb}
\usepackage{times}
\usepackage[centertags]{amsmath}
\usepackage{amsfonts}
\usepackage{amssymb}
\usepackage{amsthm}
\usepackage{newlfont}

\newtheorem{thm}{Theorem}
\newtheorem{prop}{Proposition}
\newtheorem{lem}{Lemma}
\newtheorem{remark}{Remark}
\newtheorem{definition}{Definition}

\begin{document}

\title{Generalisations of Integral Inequalities of the type of Hermite-Hadamard through convexity}

\author{Muhammad Muddassar}
\address{Department of Mathematics, University of Engineering and Technology,\\ Lahore. Pakistan}
\email{malik.muddassar@gmail.com}
\author{Muhammad Iqbal Bhatti}
\email{uetzone@hotmail.com}
\author{Wajeeha Irshad}
\email{wchattah@hotmail.com}
\date{4 June, 2012.}

\begin{abstract}
In this paper, we establish various inequalities for some differentiable mappings that are linked with the illustrious Hermite-Hadamard integral inequality for mappings whose derivatives are $s$-$(\alpha,m)$-convex.The generalised integral inequalities contribute some better estimates than some already presented. The inequalities are then applied to numerical integration and some special means.
\end{abstract}
\maketitle
{\setcounter{section}{0}}
\markboth{\underline{\hspace{3.40in}M. Muddassar, M. I. Bhatti and W. Irshad}}
{\underline{\hspace{1pt}Generalisations of Integral Inequalities of the type of Hermite-Hadamard through convexity\hspace{1.40in}}}\pagestyle{myheadings}

\section{Introduction}\label{sec1}
Let $f:\emptyset\neq I\subseteq\mathbb{R}\rightarrow\mathbb{R}$ be a function defined on the interval $I$ of real numbers. Then $f$ is called convex if
\begin{equation*}
    f(t\,x+(1-t)\,y)\leq\,\,t\,f(x)\,+\,\,(1-t)\,f(y)
\end{equation*}
for all $x,y\in I$ and $t\in [0,1]$. Geometrically, this means that if P, Q and R are three distinct points on graph of $f$ with Q between P and R, then Q is on or below chord PR. There are many results associated with convex functions in the area of inequalities, but one of those is the classical Hermite Hadamard inequality:
\begin{equation}\label{HH}
    f\left(\frac{a+b}{2}\right)\leq\frac{1}{b-a}\int_a^b f(x) dx\leq\frac{f(a)+f(b)}{2}
\end{equation}
for $a,b\in I,$ with $a<b$.\\
In \cite{r5}, H. Hudzik and L. Maligranda considered, among others, the class of functions which are $s-$convex in the first and second sense. This class is defined as follows:
\begin{definition}
A function $f:[0,\infty)\rightarrow\mathbb{R}$ is said to be $s-$convex or $f$ belongs to the class $K_s^i$
if
\begin{equation}\label{e1}
    f(\mu \,x + \nu \,y) \leq\,\, \mu^s\,f(x)\,+\, \nu^s\,f(y)
\end{equation}
holds for all $x,y \in [0,\infty),$ $\mu, \nu \in[0,1]$ and for some fixed $s\in(0,1]$.
\end{definition}
Note that, if $\mu^s + \nu^s=1$, the above class of convex functions is called $s$-convex functions in first sense and represented by $K_s^1$ and if $\mu + \nu=1$ the above class is called $s$-convex in second sense and represented by $K_s^2$.\\
It may be noted that every 1-convex function is convex. In the same paper \cite{r5} H. Hudzik and L. Maligranda discussed a few results connecting with $s-$convex functions in second sense and some new results about Hadamard's inequality for $s-$convex functions are discussed in \cite{r4}, while on the other hand there are many important inequalities connecting with 1-convex (convex) functions \cite{r4}, but one of these is $(\ref{HH}).$\\
In ~\cite{r8}, V.G. Mihesan presented the class of $(\alpha,m)$-convex functions as reproduced below:
\begin{definition}
The function $f:[0,b] \rightarrow \mathbb{R}$ is said to be $(\alpha,m)$-convex, where $(\alpha,m) \in [0,1]^2$, if for every $x,y \in [0,b]$ and $t \in [0,1]$ we have
$$f(tx + m(1 - t)y)  \leq t^\alpha f(x) + m(1 - t^\alpha)f(y)$$
\end{definition}
Note that for $(\alpha,m) \in \{(0,0) ,(\alpha,0) ,(1,0) ,(1,m) ,(1,1) ,(\alpha,1)\}$ one receives the following classes of functions respectively:  increasing, $\alpha$-starshaped, starshaped, $m$-convex, convex and $\alpha$-convex.\\
Denote by $K_m^\alpha(b)$ the set of all $(\alpha,m)$-convex functions on $[0, b]$ with $f(0) \leq 0$. For recent results and generalisations referring m-convex and  $(\alpha,m)$-convex functions see ~\cite{r1}, ~\cite{r2} and ~\cite{r14}.\\

In ~\cite{r4} S. S. Dragomir et al. discussed inequalities for differentiable and twice differentiable functions connecting with the H-H Inequality on the basis of the following Lemmas.
\begin{lem}\label{L1}
Let $f:I\subseteq\mathbb{R}\rightarrow\mathbb{R}$ be differentiable function on $I^\circ$ (interior of $I$), $a,b \in I$ with $a < b$. If $f' \in L^{1}[a,b]$, then we have
\begin{equation}\label{e3}
    \frac{f(a)+f(b)}{2}-\frac{1}{b-a}\int_a^b f(x) dx=\frac{(b-a)}{2}\int_0^1 (1-2t) f'(ta+(1-t)b)dt
\end{equation}
\end{lem}
In ~\cite{r3}, Dragomir and Agarwal established the following results connected with the right part of (\ref{e3}) as well as to apply them for some elementary inequalities for real numbers and numerical integration.
\begin{lem}\label{L2}
Let $f:I^\circ \subseteq \mathbb{R} \rightarrow \mathbb{R}$ be differentiable function on $I^\circ,$ $a,b\in I$ with $a<b$. If $f'\in L^{1}[a,b],$ then
\begin{eqnarray}\label{e4}
  &&\nonumber\!\!\!\!\!\!\!\!\!\!\!\!\!\!\frac{f(a)+f(b)}{2}-\frac{1}{b-a}\int_a^b f(x)dx \\&&  = \frac{(b-a)}{2}\int_0^1 \!\!\! \int_0^1 \left(f'(ta+(1-t)b) - f'(ua+(1-u)b)\right)(u - t)dtdu
\end{eqnarray}
\end{lem}
This paper is organized as follows. After this introduction, in section \ref{sec2} we define a generalised convex function and discuss some new integral inequalities of the type of Hermite Hadamard's for generalised convex functions. In section \ref{sec3} we give some new applications of the results from section \ref{sec2} for some special means. The inequalities are then applied to numerical integration in section \ref{sec4}.

\section{Definitions and Main Results}\label{sec2}
To establish our principal results, we first obtain the following definitions.
\begin{definition}
A function $f : [0, \infty) \rightarrow [0, \infty)$ is said to be $s$-$(\alpha,m)$-convex function in first sense or f belongs to the class ${K_{m, 1}^{\alpha, s}}$ , if for all $x, y \in [0, \infty)$ and $\mu \in  [0, 1]$, the following inequality holds:
\begin{equation*}
    f(\mu x + (1-\mu) y) \leq \left({\mu^\alpha}^s \right) f(x) + m \left(1-{\mu^\alpha}^s \right) f\left(\frac{y}{m}\right)
\end{equation*}
where $(\alpha,m) \in [0,1]^2$ and for some fixed $s \in (0, 1]$.
\end{definition}
\begin{definition}
A function $f : [0, \infty) \rightarrow [0, \infty)$ is said to be $s$-$(\alpha,m)$-convex function in second sense or f belongs to the class ${K_{m, 2}^{\alpha, s}}$ , if for all $x, y \in [0, \infty)$ and $\mu, \nu \in  [0, 1]$, the following inequality holds:
\begin{equation*}
    f(\mu x + (1-\mu) y) \leq \left(\mu^\alpha \right)^s f(x) + m \left(1-\mu^\alpha \right)^s f\left(\frac{y}{m}\right)
\end{equation*}
where $(\alpha,m) \in [0,1]^2$ and for some fixed $s \in (0, 1]$.
\end{definition}
\begin{thm}\label{T1}
Let $f:I^o\subseteq \mathbb{R}\rightarrow \mathbb{R}$ be a differentiable function on $I^o$ (interior of I), $a,b \in I$ with $a<b$. If $f' \in L^1[a,b]$. If the mapping $|f'|$ is $s$-$(\alpha,m)$-convex on $[a,b]$, then
\begin{equation}\label{Te1}
\left|\frac{f(a)+f(b)}{2}-\frac{1}{b-a}\int_a^b\,f(x)dx\right|\leq \frac{(b-a)}{2}\left[ v_1|f'(a)|+ v_2\left|f'\left(\frac{b}{m}\right)\right|\right]
\end{equation}
where $v_1=\frac{1+2^{\alpha s}(\alpha s)}{ 2^{\alpha s}(\alpha s+1)(\alpha s+2)}$ and $v_2=m\left(\frac{1}{2}- v_1\right)$
\end{thm}
\begin{proof}
Taking modulus on both sides of  lemma \ref{L1}, we get
\begin{equation}\label{t1a}
\left|\frac{f(a)+f(b)}{2}-\frac {1}{b-a} \int^b_a f(x) dx\right|\leq\frac{b-a}{2} \int _0^1 |(1-2t)||f'(ta+(1-t)b)|dt
\end{equation}
Since $|f'|$ is  $s$-$(\alpha,m)$-  convex on $[a,b]$ for all $t\in [0,1]$, then the above inequality becomes
\begin{eqnarray}\label{t1b}
 &&\nonumber \!\!\!\!\!\!\!\!\!\!\!\!\!\!\!\!\!\!\!\! \left|\frac{f(a)+f(b)}{2}-\frac{1}{b-a}\int_a^b f(x)dx\right| \\&& \indent\indent\leq \frac{(b-a)}{2} \int_0^1|1-2t|\left[t^{\alpha s}|f'(a)|+m(1-t^{\alpha s})\left|f'\left(\frac{b}{m}\right)\right|\right]dt
 \end{eqnarray}
Here
\begin{equation}\label{t1c}
 \int_0^1 t^{\alpha s} |1-2t|dt=\int_0^{\frac{1}{2}} (1-2t){t^\alpha}^s dt+\int_{\frac{1}{2}}^1 (2t-1){t^\alpha}^s dt= \frac{1+2^{\alpha s}(\alpha s)}{ 2^{\alpha s}(\alpha s+1)(\alpha s+2)}
 \end{equation}
and
\begin{equation}\label{t1d}
\int_0^1 (1-t^{\alpha\,s})|1-2t|\,dt=\frac{1}{2}- \frac{1+2^{\alpha\,s}\,(\alpha\,s)}{ 2^{\alpha\,s}(\alpha s+1)(\alpha s+2)}
\end{equation}
Inequations (\ref{t1b}), (\ref{t1c}) and (\ref{t1d}) together imply (\ref{Te1}).
\end{proof}
\begin{remark}
For $(\alpha,m)$=$(1,1)$ in (\ref{Te1}), we get Theorem 2 of ~\cite{r10}.
\end{remark}
%%%%%%%%%%%%%%%%%%%%%%%%%%%%%%%%%%%%%%%%%%%%%%%%%%%%%%%%%%%%%%%%%%%%%%%%%%%%%%%%
\begin{thm}\label{T2}
Let the assumptions of theorem \ref{T1} be satisfied with $p>1$, such that $q=\,\frac{p}{p-1}$. If the mapping $|f'|^q$ is $s$-$(\alpha,m)$-convex on $[a,b]$, then
\begin{equation}\label{Te2}
\left|\frac{f(a)+f(b)}{2}-\frac{1}{b-a}\int_a^b f(x)dx\right|
\leq \frac{b-a}{2(p+1)^{\frac{1}{p}}}\left[\frac{|f'(a)|^q +m\alpha s|f'(\frac{b}{m})|^q}{\alpha s+1}\right]^{\frac{1}{q}}
\end{equation}
\end{thm}
\begin{proof}
By applying H\"{o}lder's inequality on the right side of (\ref{t1a}), we have
\begin{equation}\label{t2a}
\!\!\int_0^1\!\!|1-2t||f'(ta+(1-t)b)|dt\!\leq\!\left(\!\int_0^1\!\!|1-2t|^p dt\!\right)^\frac{1}{p}\!\!\left(\!\int_0^1\!\!|f'(ta+(1-t)b)|^q dt\!\right)^\frac{1}{q}
\end{equation}
here
\begin{equation}\label{t2b}
\int _0^1 |1-2t|^p dt=\frac{1}{1+p}
\end{equation}
since $|f'|^q$ is  $s$-$(\alpha,m)$-convex on $[a,b]$ for all $t\in [0,1]$, therefore
\begin{equation*}
|f'(ta+(1-t)b)|^q \leq {t^\alpha}^s\left|f'(a)\right|^q+ m(1-{t^\alpha}^s)\left|f'(b)\right|^q
\end{equation*}
So the second integral on right side of (\ref{t2a}) can be simplified by simple integration as:
\begin{equation}\label{t2c}
\int_0^1 |f'(ta+(1-t)b)|^q dt= \left[\frac{|f'(a)|^q+m \alpha s|f'(\frac{b}{m})|^q}{\alpha s+1}\right]
\end{equation}
Inequations (\ref{t2a}), (\ref{t2b}) and (\ref{t2c}) together imply (\ref{Te2}).
\end{proof}
\begin{remark}
For $(\alpha,m)$=$(1,1)$ in (\ref{Te2}), we get Theorem 4 of ~\cite{r10}.
\end{remark}
%%%%%%%%%%%%%%%%%%%%%%%%%%%%%%%%%%%%%%%%%%%%%%%%%%%%%%%%%%%%%%%%%%%%%%%%%%%%%%%%
\begin{thm}\label{T3}
Let the assumptions of Theorem \ref{T1} be satisfied with $q>1$. If the mapping $|f'|^q$ is $s$-$(\alpha,m)$-convex on $[a,b]$, then
\begin{equation}\label{Te3}
\left|\frac{f(a)+f(b)}{2}-\frac{1}{b-a}\int_a^b f(x)dx\right|\leq \frac{(b-a)}{2^\frac{p+1}{p}}\left[v_1|f'(a)|^q+v_2 \left|f'\left(\frac{b}{m}\right)\right|^q\right]^{\frac{1}{q}}
\end{equation}
where $v_1=\frac{1+2^{\alpha s}(\alpha s)}{ 2^{\alpha s}(\alpha s+1)(\alpha s+2)}$ and $v_2=m\left(\frac{1}{2}- v_1\right)$
\end{thm}
\begin{proof}
Inequation (\ref{t1a}) reduces to the following form:
\begin{eqnarray}\label{t3a}
&&\nonumber\!\!\!\!\!\!\!\!\!\!\!\!\!\!\left|\frac{f(a)+f(b)}{2}-\frac{1}{b-a}\int_a^b f(x)dx\right| \\&& \indent\indent\indent\leq \frac{(b-a)}{2}\int _0^1|1-2t|^{\frac{1}{p}} |1-2t|^{\frac{1}{q}} |f'(ta+(1-t)b)|dt
\end{eqnarray}
where $\frac{1}{p}+\frac{1}{q}=1$.\\
By applying H\"{o}lder's inequality on (\ref{t3a}), for $q>1$ we have
\begin{eqnarray}\label{t3b}
&&\nonumber\!\!\!\!\!\!\!\!\!\!\!\!\!\!\!\!\!\left|\frac{f(a)+f(b)}{2}-\frac{1}{b-a}\int_a^b f(x)dx\right| \\&& \indent \leq \!\frac{(b-a)}{2}\!\left(\!\!\int_0^1 \!\!|1-2t|dt\!\right)^\frac{1}{p}\! \left(\!\int_0^1\!\! |1-2t||f'(ta+(1-t)b)|^q dt\!\right)^\frac{1}{q}
\end{eqnarray}
Applying the $s$-$(\alpha,m)$ convexity of $|f'|^q$ on $[a,b]$ for all $t\in [0,1]$ on the second integral on the right side of (\ref{t3b}) , we have
\begin{eqnarray}\label{t3c}
&&\nonumber\!\!\!\!\!\!\!\!\!\!\!\!\!\!\!\!\!\left|\frac{f(a)+f(b)}{2}-\frac{1}{b-a}\int_a^bf(x)dx\right| \\&& \!\!\!\!\!\! \leq\!\frac{(b-a)}{2}\!\left(\frac{1}{2}\right)^{\frac{1}{p}}\!\!
\left(\!t^{\alpha s} |1-2t| |f'(a)|^q +m (1-t^{\alpha s})|1-2t|\left|f'\left(\!\frac{b}{m}\right)\right|^q \!\!dt\!\right)^{\frac{1}{q}}
\end{eqnarray}
Here
\begin{equation}\label{t3d}
\int_0^1 {t^\alpha}^s|1-2t|dt= \frac{1+{2^\alpha}^s (\alpha s)}{ {2^\alpha}^s (\alpha s+1)(\alpha s+2)}
\end{equation}
and in similar manner
\begin{equation}\label{t3e}
\int_0^1 (1-t^{\alpha\,s})|1-2t|\,dt=\frac{1}{2}- \frac{1+2^{\alpha\,s}\,(\alpha\,s)}{ 2^{\alpha\,s}(\alpha s+1)(\alpha s+2)}
\end{equation}
Inequations (\ref{t3c}), (\ref{t3d}) and (\ref{t3e}) together imply (\ref{Te3}).
\end{proof}
\begin{remark}
For $(\alpha,m)$=$(1,1)$ in (\ref{Te3}), we get Theorem 6 of ~\cite{r10}.
\end{remark}
%%%%%%%%%%%%%%%%%%%%%%%%%%%%%%%%%%%%%%%%%%%%%%%%%%%%%%%%%%%%%%%%%%%%%%%%%%%%%%%%
\begin{thm}\label{T4}
Let $f:I^o\subseteq \mathbb{R}\rightarrow \mathbb{R}$ be a differentiable function on $I^o$ (interior of I), $a,b \in I$ with $a<b$. If $f'\in L^1[a,b]$. If the mapping $|f'|$ is $s$-$(\alpha,m)$-convex on $[a,b]$, then
\begin{equation}\label{Te4}
\left|\frac{f(a)+f(b)}{2}-\frac{1}{b-a}\int_a^b f(x)dx\right|\leq \frac{(b-a)}{2}\left[ u_1|f'(a)|  + u_2 \left|f'\left(\frac{b}{m}\right)\right|\right]
\end{equation}
where $u_1=\frac{(\alpha s)^2 +3\alpha s+4}{2 (\alpha s+1)(\alpha s+2) (\alpha s+3)}$ and $u_2=m\left(\frac{1}{3}-u_1\right)$
\end{thm}
\begin{proof}
Taking absolute value of Lemma \ref{L2}, we get
\begin{eqnarray}\label{t4a}
&&\nonumber\!\!\!\!\!\!\!\!\!\!\!\!\!\!\!\!\!\!\!\!\!\!\!\!\!\!\left|\frac{f(a)+f(b)}{2}-\frac{1}{b-a}\int_a^b f(x)dx \right| \leq \frac{(b-a)}{2}\int_0^1\!\!\int_0^1 \left|f'(ta+(1-t)b) \right.\\&& \nonumber\indent\indent\indent\indent\indent\indent\indent\indent\indent\indent\indent\indent\indent \left.-f'(ua+(1-u)b)\right||u-t|dt du\\
&&\indent\indent\indent\indent\indent\indent\indent\indent = (b-a)\int_0^1\!\!\int_0^1 \!|f'(ta+(1-t)b)|u-t|dt du
\end{eqnarray}
Since $|f'|$ is $s$-$(\alpha,m)$-convex on $[a,b]$ for all $t\in [0,1]$, so the above inequation (\ref{t4a}) may be written as
\begin{eqnarray}\label{t4b}
&&\nonumber\!\!\!\!\!\!\!\!\!\!\!\!\!\!\left|\frac{f(a)+f(b)}{2}-\frac{1}{b-a}\int_a^b f(x)dx\right|\\
&&\!\leq\!\!(b-a)\!\!\int_0^1\!\!\!\int_0^1\!\!\! \left(\! {t^ \alpha}^ s |u-t||f'(a)|+m (1-{t^\alpha}^s) |u-t|\left|f'\left(\frac{b}{m}\right)\right|\!\right)dtdu
\end{eqnarray}
Here
\begin{eqnarray}\label{t4c}
&&\nonumber\!\!\!\!\!\!\!\!\!\!\!\!\!\int_0^1\int_0^1 {t^\alpha}^s |u-t|dtdu=\int_0^1\left\{\int_0^u {t^\alpha}^s (u-t)dt+\int_u^t {t^\alpha}^s (t-u)dt\right\}du\\&&\indent\indent\indent\indent\indent\indent=\frac{(\alpha s)^2 +3\alpha s+4}{2 (\alpha s+1)(\alpha s+2) (\alpha s+3)}
\end{eqnarray}
and analogously
\begin{equation}\label{t4d}
\int_0^1 \int_0^1 (1- {t^\alpha}^s)|u-t|dtdu=\frac{1}{3}-\frac{(\alpha s)^2 +3\alpha s+4}{2 (\alpha s+1)(\alpha s+2) (\alpha s+3)}
\end{equation}
Inequations (\ref{t4b}), (\ref{t4c}) and (\ref{t4d}) together imply (\ref{Te4}).
\end{proof}
\begin{remark}
For $(\alpha,m)$=$(1,1)$ in (\ref{Te4}), we get Theorem 8 of ~\cite{r10}.
\end{remark}
%%%%%%%%%%%%%%%%%%%%%%%%%%%%%%%%%%%%%%%%%%%%%%%%%%%%%%%%%%%%%%%%%%%%%%%%%%%%%%%%
\begin{thm}\label{T5}
Let the assumptions of theorem \ref{T4} be satisfied with $p>1$, such that $q=\frac{p}{p-1}$. If the mapping $|f'|^q$ is $s$-$(\alpha,m)$-convex on $[a,b]$, then
\begin{eqnarray}\label{Te5}
&&\nonumber\!\!\!\!\!\!\!\!\!\!\!\!\!\!\!\!\!\!\!\!\!\!\!\!\!\!\!\left|\frac{f(a)+f(b)}{2}-\frac{1}{b-a}\int_a^bf(x)dx\right|\leq (b-a)\left[\frac{2}{(p+1)(p+2)}\right]^{\frac{1}{p}}\times\\
&&\indent\indent\indent\indent\indent\indent\indent\indent\indent\indent\indent\left(\frac{|f'(a)|^q + m\alpha s|f'(\frac{b}{m})|^q}{\alpha s+1}\right)^{\frac{1}{q}}
\end{eqnarray}
\end{thm}
\begin{proof}
By applying H\"{o}lder's inequality on the right side of (\ref{t1a}), we have
\begin{eqnarray}\label{t5a}
&&\nonumber\!\!\!\!\!\!\!\!\!\!\!\!\left|\frac{f(a)+f(b)}{2}-\frac{1}{b-a}\int_a^b f(x)dx\right|\!\leq\! (b-a)\!\left(\!\!\int_0^1\!\!\!\int_0^1 |f'(ta+(1-t)b)|^q dt du\!\right)^{\frac{1}{q}}\!\!\times\\
&&\indent\indent\indent\indent\indent\indent\indent\indent\indent\indent\indent\indent\indent\indent\indent\indent\left(\int_0^1\!\!\!\int_0^1 \!\! |u-t|^pdtdu\right)^{\frac{1}{p}}
\end{eqnarray}
here
\begin{equation}\label{t5b}
\int_0^1\int_0^1 |u-t|^p dt du=\frac{2}{(p+1)(p+2)}
\end{equation}
since  by $s$-$(\alpha,m)$ convexity of $|f'|^q$ on $[a,b]$ for all $t\in [0,1]$. The first integral on the right side of (\ref{t5a}) may be solved as
\begin{equation}\label{t5c}
\int_0^1 \int_0^1 |f'(ta+(1-t)b)|^q dt du \leq \frac{\left|f'(a)\right| + m\alpha s \left|f'\left(\frac{b}{m}\right)\right|}{\alpha s + 1}
\end{equation}
Inequations (\ref{t5a}), (\ref{t5b}) and (\ref{t5c}) together imply (\ref{Te5}).
\end{proof}
\begin{remark}
For $(\alpha,m)$=$(1,1)$ in (\ref{Te5}), we get Theorem 10 of ~\cite{r10}.
\end{remark}
%%%%%%%%%%%%%%%%%%%%%%%%%%%%%%%%%%%%%%%%%%%%%%%%%%%%%%%%%%%%%%%%%%%%%%%%%%%%%%%%
\begin{thm}\label{T6}
Let the assumptions of Theorem \ref{T4} are satisfied with $q>1$. If the mapping $|f'|^q$ is $s$-$(\alpha,m)$-convex on $[a,b]$, then
\begin{equation}\label{Te6}
\left|\frac{f(a)+f(b)}{2}-\frac{1}{b-a}\int_a^b f(x)dx\right| \leq \frac{(b-a)}{3^{1/p}}\left[u_1 |f'(a)|^q + u_2 \left|f'\left(\frac{b}{m}\right)\right|^q \right]^{\frac{1}{q}}
\end{equation}
where $u_1=\frac{(\alpha s)^2 +3\alpha s+4}{2 (\alpha s+1)(\alpha s+2) (\alpha s+3)}$ and $u_2=m\left(\frac{1}{3}-u_1\right)$
\end{thm}
\begin{proof}
By applying H\"{o}lder's inequality on the right side of (\ref{t1a}), we have
\begin{eqnarray}\label{t6a}
&&\nonumber\!\!\!\!\!\!\!\!\!\!\!\!\!\!\!\!\!\left|\frac{f(a)+f(b)}{2}-\frac{1}{b-a}\int_a^b f(x)dx\right|\!\leq\! (b-a)\!\left(\!\!\int_0^1\!\!\!\int_0^1\!\!\!|u-t| |f'(ta+(1-t)b)|^q dt du\!\right)^{\frac{1}{q}}\!\!\times\\
&&\indent\indent\indent\indent\indent\indent\indent\indent\indent\indent\indent\indent\indent\indent\indent\indent\indent\indent\left(\!\int_0^1 \!\!\int_0^1\!\! |u-t| dt du\!\right)^{\frac{1}{p}}
\end{eqnarray}
here
\begin{equation}\label{t6b}
\int_0^1\int_0^1 |u-t| dt du= \frac{1}{3}
\end{equation}
since  by $s$-$(\alpha,m)$ convexity of $|f'|^q$ on $[a,b]$ for all $t\in [0,1]$. The first integral on the right side of (\ref{t6a}) may be solved as
\begin{eqnarray}\label{t6c}
&&\nonumber\!\!\!\!\!\!\!\!\!\!\!\!\!\!\!\int_0^1 \int_0^1 |u-t||f'(ta+(1-t)b)|^q dt du \\
&&\indent\indent\leq\! \int_0^1\!\! \int_0^1 \!\!|u-t|\left({t^\alpha}^s|f'(a)|^q + m(1-{t^\alpha}^s)\left|f'\left(\frac{b}{m}\right)\right|^q\right) dtdu
\end{eqnarray}
and
\begin{equation}\label{t6d}
\int_0^1\int_0^1  {t^\alpha}^s |u-t|dt du=\frac{{\alpha s}^2 +3\alpha s +4}{2(\alpha s+1)(\alpha s+2) (\alpha s+3)}
\end{equation}
and
\begin{equation}\label{t6e}
\int_0^1 \int_0^1 (1- {t^\alpha}^s)|u-t|dt du=\frac{1}{3}- \frac{{\alpha s}^2 +3\alpha s +4}{2(\alpha s+1)(\alpha s+2) (\alpha s+3)}
\end{equation}
which completes the proof.
\end{proof}
\begin{remark}
For $(\alpha,m)$=$(1,1)$ in (\ref{Te6}), we get Theorem 12 of ~\cite{r10}.
\end{remark}
%%%%%%%%%%%%%%%%%%%%%%%%%%%%%%%%%%%%%%%%%%%%%%%%%%%%%%%%%%%%%%%%%%%%%%%%%%%%%%%%
\section{Application to some special means}\label{sec3}
Let us recall the following means for any two positive numbers $a$ and $b$.
\begin{enumerate}
  \item  \textit{The Arithmetic mean}
  $$A\equiv A(a,b)=\frac{a+b}{2}$$
  \item \textit{The Harmonic mean}
  $$H\equiv H(a,b)=\frac{2ab}{a+b}  $$
  \item \textit{The $p-$Logarithmic mean}\\
  $$L_p\equiv L_{p}(a,b)=\left\{
                           \begin{array}{ll}
                             a, & \hbox{if $a=b$;}   \\
                             \left[\frac{b^{p+1}-a^{p+1}}{(p+1)(b-a)}\right]^\frac{1}{p}, & \hbox{if $a\neq b$.}
                           \end{array}
                         \right.$$
  \item The $Identric\ \ mean$\\
  $$I\equiv I(a,b)=\left\{
                           \begin{array}{ll}
                             a, & \hbox{if $a=b$;}  \\
                             \frac{1}{e}\left(\frac{b^b}{a^a}\right)^\frac{1}{b-a}, & \hbox{if $a\neq b$.}
                           \end{array}
                         \right.$$
 \item \textit{The Logarithmic mean}\\
  $$L\equiv L(a,b)=\left\{
                           \begin{array}{ll}
                             a, & \hbox{if $a=b$;}   \\
                             \frac{b-a}{\log b\ -\ \log a}, & \hbox{if $a\neq b$.}
                           \end{array}
                         \right.$$
\end{enumerate}
The following inequality is well known in the literature in ~\cite{r9}:
$$H\leq G \leq L\leq I\leq A.$$
It is also known that $L_p$ is  monotonically increasing over $p\in\mathbb{R},$ denoting $L_0=I$ and $L_{-1}=L.$
\begin{prop}\label{P1}
Let $p>1,$ $0<a<b$ and $q=\frac{p}{p-1}.$ Then one has the inequality.
\begin{equation}\label{S1}
\left| \ A(a,b)\ - \ L(a,b)\right| \leq \frac{\log b-\log a}{2(p+1)^{1/p}} A^{\frac{1}{q}} \left(|a|^q, |b|^q\right).
\end{equation}
\end{prop}
\begin{proof} By Theorem \ref{T2} applied for the mapping $f(x)=e^x$ for $s$-$(\alpha,m)$=$1$-$(1,1)$, we have the above inequality (\ref{S1}).
\end{proof}\\
Another result which is connected with $p-$Logarithmic mean $L_{p}(a,b)$ is the following one:
\begin{prop}\label{P2}
Let $p>1,$ $0<a<b$ and $q=\frac{p}{p-1},$ then
\begin{equation*}
\left|\frac{\mathrm{I}(a, b)}{\mathrm{G}(a, b)}\right| \leq \exp \left(\frac{b-a}{2} \,\,\,\mathrm{H}^{-1/q}\left(|a|^q,|b|^q\right)\right)
\end{equation*}
\end{prop}
\begin{proof} Follows by Theorem \ref{T3}, setting $f(x)=- \log(1-x)$ for $s$-$(\alpha,m)$=$1$-$(1,1)$.
\end{proof}\\
Another result which is connected with $p-$Logarithmic mean $L_{p}(a,b)$ is the following one:
\begin{prop}\label{P3}
Let $p>1,$ $0<a<b$ and $q=\frac{p}{p-1},$ then
\begin{eqnarray*}
&&\!\!\!\!\!\!\!\!\!\!\!\!\!\!\!\! \left|A\left[a^n, b^n\right] - L_n^n\left[a^n, b^n\right]\right|  \leq |n|^q\frac{b-a}{3}\left[\ A \left( \left|a\right|^{q(n-1)},\left|b\right|^{q(n-1)}\right)\right].
\end{eqnarray*}
\end{prop}
\begin{proof} Follows by Theorem \ref{T6}, setting $f(x)=(1-x)^n$, $|n| \geq 2$ and $n \in \mathbb{Z}$ for $s$-$(\alpha,m)$=$1$-$(1,1)$.
\end{proof}
\section{Error Estimates for Trapezoidal Formula}\label{sec4}
Let $D$ be the partition  $\{a= x_0 < x_1 < ... < x_{n-1} < x_n = b\}$ of the interval $[a, b]$ and consider the quadrature formula
\begin{equation}\label{m1}
\int_a^b f(x)dx = S(f, D) + R(f, D)
\end{equation}
where
\begin{equation*}
S(f, D)=\sum_{k=0}^{n-1} \frac{f(x_k)+f(x_{k+1})}{2}\left(x_{k+1}-x_k\right)
\end{equation*}
for the trapezoidal version and $R(f,D)$ denotes the related approximation error.
\begin{prop}\label{pro1}
Let $f: I \subseteq \mathbb{R} \rightarrow \mathbb{R}$ be a differentiable mapping on $I^o$ such  that $f' \in L^1[a, b]$, where $a, b \in I$ with $a < b$ and $|f'|$ is $s$-convex on $[a, b]$, for every partition $D$ of $[a, b]$ the trapezoidal error approximate satisfies
\begin{eqnarray}\label{pr1}
&&\!\!\!\!\!\!\!\!\!\!\!\!\!\!\!\!\!\!\!\!\!\left|R(f, D)\right|\! \leq \! \frac{1}{2^{\frac{1}{p}}}\!\left(\!\frac{s.2^s+1}{2^s(s+1)(s+2)}\! \right)^{\frac{1}{q}}\!\sum_{k=0}^{n-1}\! \frac{\left(x_{k+1}-x_k\right)^2}{2}\!\left[\left|f'(x_k)\right|\!+\!\left|f'(x_{k+1})\right|\right]
\end{eqnarray}
where $p>1$
\end{prop}
\begin{proof}
By applying Theorem $3$ on the subinterval $[x_k, x_(k+1)]$ of the partition $D$ of $[a, b]$  for  $k=0, 1, 2, …, n-1$, for $s$-$(\alpha,m)$=$s$-$(1,1)$ and using the fact: $\sum_{m=1}^{n-1} \left(\Phi_m+\Psi_m \right)^r \leq \sum_{m=1}^{n-1} (\Phi_m)^r +\sum_{m=1}^{n-1} (\Psi_m)^r$   for $(0<r<1)$ and for each $m$ both $\Phi_m,\Psi_m \geq 0$ , we have
\begin{eqnarray}\label{pre1}
&&\nonumber\!\!\!\!\!\!\!\!\!\!\!\!\!\!\!\!\!\!\!\!\!\!\!\left|f\left(\frac{x_{k+1}+x_k}{2}\right) - \frac{1}{x_{k+1} - x_k} \int_{x_k}^{x_{k+1}}f(x)dx \right| \\&&  \indent\indent\indent\indent\leq\!\frac{(x_{k+1}-x_k)}{2^{\frac{p+1}{p}}}\!\left(\!\frac{s.2^s+1}{2^s(s+1)(s+2)}\!\right)^{\frac{1}{q}}\!\!\left(\left|f'(x_k)\right|\!+\!\left|f'(x_{k+1})\right|\right)
\end{eqnarray}
Taking sum over $k$ from $0$ to $n-1$ and taking into account that $|f'|^q$ is $s$-$(\alpha,m)$-convex, we get
\begin{eqnarray*}
&&\!\!\!\!\!\!\!\!\!\!\!\!\!\!\!\left|\int_a^b \!\!f(x)dx - \!S(f, K)\right|  =\left|\!\sum_{k=0}^{n-1}\left\{\!\!\int_{x_k}^{x_{k+1}}f(x)dx\! - \! \left(\!x_{k+1}-x_k\right)\!\frac{f(x_{k+1})+f(x_k)}{2}\right\}\right|\\&&\indent\indent\indent\indent\indent\indent\nonumber \leq \!\sum_{k=0}^{n-1}\left|\left\{\!\!\int_{x_k}^{x_{k+1}}f(x)dx\! - \!\!\left(x_{k+1}-x_k\right) \frac{f(x_{k+1})+f(x_k)}{2}\right\}\right|
\end{eqnarray*}
This gives
\begin{eqnarray}\label{pre2}
&&\!\!\!\!\!\!\!\!\!\!\!\!\!\!\!\!\!\!\left|R(f,D)\right|  \leq  \!\sum_{k=0}^{n-1}\left(x_{k+1}-x_k\right)\left|\frac{f(x_{k+1})+f(x_k)}{2} - \!\frac{1}{\left(x_{k+1}-x_k\right)}\!\int_{x_k}^{x_{k+1}}\!\!f(x)dx\right|
\end{eqnarray}
By combining (\ref{pre1}) and (\ref{pre2}), we get (\ref{pr1}). Which completes the proof.
\end{proof}
\begin{prop}\label{pro2}
Let $f: I \subseteq \mathbb{R} \rightarrow \mathbb{R}$ be a differentiable mapping on $I^o$ such  that $f' \in L^1[a, b]$, where $a, b \in I$ with $a < b$ and $|f'|^q$ is $s$-$(\alpha,m)$-convex on $[a, b]$, then for every partition $D$ of $[a, b]$ the trapezoidal error approximate satisfies
\begin{eqnarray}
&&\!\!\!\!\!\!\!\!\!\!\!\!\!\!\nonumber\left|R(f, D)\right| \leq \left(\frac{2}{3}\right)^{\frac{1}{p}}\!\left(\!\frac{s^2+3s+4}{(s+1)(s+2)(s+3)}\!\right)^{\!\frac{1}{q}}\!\sum_{k=0}^{n-1} \frac{\left(x_{k+1}-x_k\right)^2}{2}\left[\left|f'(x_k)\right|^q+\left|f'(x_{k+1})\right|^q\right]^{\frac{1}{q}} \\&&
\indent \leq \!\left(\frac{2}{3}\right)^{\frac{1}{p}}\!\left(\!\frac{s^2+3s+4}{(s+1)(s+2)(s+3)}\!\right)^{\!\frac{1}{q}}\!\sum_{k=0}^{n-1} \frac{\left(x_{k+1}-x_k\right)^2}{2}\left[\left|f'(x_k)\right| +\left|f'(x_{k+1})\right|\right].
\end{eqnarray}
\end{prop}
\begin{proof}
Proof is very similar as that of Proposition \ref{pro1}  by using Theorem \ref{T6}.
\end{proof}
\section{Acknowledgement}
This research paper is made possible through the help and support from everyone, including: parents, teachers, family, friends, and in essence, all sentient beings.
Especially, please allow us to dedicate my acknowledgment of gratitude toward the following significant advisors and contributors:\\
First and foremost, we would like to thank prof. Dr. A. D. Raza Choudary (Director General, ASSMS, Lahore) for his most support and encouragement. He kindly read my paper and offered invaluable detailed
advices on grammar, organization, and the theme of the paper.\\
Finally, I gratefully acknowledge the time and expertise devoted to reviewing papers by the advisory editors, the members of the editorial board, and the referees.

\end{document}